\documentclass{amsart}

\usepackage{amsfonts,amsmath,amssymb} 
\usepackage[mathcal]{euscript}

\usepackage{graphicx}
\usepackage[T1]{fontenc}
\usepackage{enumerate}
\usepackage{tikz}
\usetikzlibrary{calc}
\usepackage{textcomp}  
\newtheorem{theorem}{Theorem}[section]
\newtheorem{lemma}{Lemma}[section]

\newtheorem{corollary}{Corollary}[section]

\theoremstyle{definition}

\newtheorem{remark}{Remark}[section]
\newtheorem{question}{Question}[section]

\newtheorem{example}{Example}[section]
\newcommand{\Sp}[1]{\mathcal{#1}}
\newcommand{\Spint}[1]{\Sp{#1}^{\circ}}

\DeclareMathOperator{\arcosh}{arcosh}
\DeclareMathOperator{\arsinh}{arsinh}
\DeclareMathOperator{\artanh}{artanh}
\DeclareMathOperator{\diam}{diam}
\DeclareMathOperator{\Si}{Si}
\DeclareMathOperator{\SHi}{SHi}
\DeclareMathOperator{\ASi}{ASi}
\DeclareMathOperator{\ASHi}{ASHi}
\DeclareMathOperator{\Ti}{Ti}
\DeclareMathOperator{\THi}{THi}
\DeclareMathOperator{\ATi}{ATi}
\DeclareMathOperator{\ATHi}{ATHi}
\DeclareMathOperator{\id}{id}
\begin{document}
\title{On Seiffert-like means}

\author{Alfred Witkowski}
\address{Institute of Mathematics and Physics\\University of Technology and Life Sciences\\al. prof. Kaliskiego 7\\85-796 Bydgoszcz, Poland}
\email{alfred.witkowski@utp.edu.pl}
\date{August 7, 2013}
\begin{abstract}
We investigate the representation of homogeneous, symmetric means in the form 
$$M(x,y)=\frac{x-y}{2f\left(\frac{x-y}{x+y}\right)}.$$
This allows for a new approach to comparing means. As an example, we provide optimal estimate of the form
\begin{equation*}
	({1-\mu}){\min(x,y)}+ {\mu}{\max(x,y)}\leqslant{M(x,y)}\leqslant ({1-\nu}){\min(x,y)}+ {\nu}{\max(x,y)}
\end{equation*}
and
\begin{equation*}
M(\tfrac{x+y}{2}-\mu\tfrac{x-y}{2},\tfrac{x+y}{2}+\mu\tfrac{x-y}{2})\leqslant N(x,y)\leqslant M(\tfrac{x+y}{2}-\nu\tfrac{x-y}{2},\tfrac{x+y}{2}+\nu\tfrac{x-y}{2})
\end{equation*} 
for some known means.

\end{abstract}
\maketitle
\section{Introduction, Definitions and  notation}

Looking at the two means introduced by Seiffert in \cite{SE3} 
\begin{equation*}
	P(x,y)=\begin{cases}
		\dfrac{x-y}{2\arcsin\frac{x-y}{x+y}}\label{def:eq_se1} & x \neq y\\
				x & x=y
	\end{cases},
\end{equation*}
and in \cite{SE4}
\begin{equation*}
	T(x,y)=\begin{cases}
		\dfrac{x-y}{2\arctan\frac{x-y}{x+y}}\label{def:eq_se2} & x \neq y\\
				x & x=y
	\end{cases},
\end{equation*}
as well as   two other Seiffert-like means involving inverse hyperbolic function: introduced  in \cite{NeuSan}
\begin{equation*}
	M(x,y)=\begin{cases}
		\dfrac{x-y}{2\arsinh\frac{x-y}{x+y}}\label{def:eq_se31} & x \neq y\\
				x & x=y
	\end{cases},
\end{equation*}
and well known logarithmic mean
\begin{equation*}
	L(x,y)=\begin{cases}
		\dfrac{x-y}{2\artanh\frac{x-y}{x+y}}=\dfrac{x-y}{\log x-\log y}\label{def:eq_se4} & x \neq y\\
				x & x=y
	\end{cases},
\end{equation*}
it begets the following questions:
\begin{question}\label{qu:1}For positive $x,y$ and a positive function $f:(0,1)\to\mathbb{R}$ let 

\begin{equation}\label{eq:definition of S_f}
	S_f(x,y)=\begin{cases}
		\dfrac{|x-y|}{2f\left(\frac{|x-y|}{x+y}\right)} & x\neq y,\\
		x & x=y
	\end{cases}. 
\end{equation}
Under what assumptions on $f$ the function $S_f$ is a mean, i.e. satisfies $\min(x,y)\leqslant S_f(x,y)\leqslant\max(x,y)$ for all $x,y>0$?
\end{question}
\begin{question}\label{qu:2}
What means $M$ can be represented in the form $S_{f}$?
\end{question}

The aim of this paper is to answer  the two question stated above and explore the subject. In particular, we introduce a metric and an algebraic structure on the set of means.
Then we show how the representation of a mean in the form $S_{f}$ can be used to investigate properties of $M$. In particular, we offer a simple criterium for Schur convexity of $M$ and  criteria for finding optimal bounds of the form
\begin{align*}
	({1-\mu}){K(x,y)}+ {\mu}{N(x,y)}&\leqslant{M(x,y)}\leqslant ({1-\nu}){K(x,y)}+ {\nu}{N(x,y)},\\
M(\tfrac{x+y}{2}-\mu\tfrac{x-y}{2},\tfrac{x+y}{2}+\mu\tfrac{x-y}{2})&\leqslant N(x,y)\leqslant M(\tfrac{x+y}{2}-\nu\tfrac{x-y}{2},\tfrac{x+y}{2}+\nu\tfrac{x-y}{2})
	\end{align*}
where $K<M<N$ are homogeneous means.

Denote by $\Sp{M}$ the set of functions $M:\mathbb{R}_+^2\to\mathbb{R}$ satisfying the following conditions:
\begin{enumerate}[(A)]
	\item $M$\label{mean_prop:A} is symmetric, i.e. $M(x,y)=M(y,x)$ for all $x,y\in\mathbb{R}_+$.
	\item $M$ \label{mean_prop:B}  is positively homogeneous of order $1$, i.e. for all $\lambda>0$ holds $M(\lambda x,\lambda y)=\lambda M(x,y)$.
	\item $M$ \label{mean_prop:C} "lies in between", i.e. $\min(x,y)\leqslant M(x,y)\leqslant \max(x,y)$.
\end{enumerate}
 Elements of $\Sp{M}$ shall be called means. 
 The set of strict means, i.e. means  satisfying the condition 
\begin{enumerate}[(A)]
\setcounter{enumi}{3}
	\item \label{mean_prop:D} $\min(x,y)< M(x,y)< \max(x,y)$ whenever $x\neq y$,
\end{enumerate}
 will be denoted by $\Spint{M}$.

By $\Sp{S}$ we shall denote the set of   functions $f:(0,1)\to\mathbb{R}$ satisfying
\begin{equation}\label{ineq:basic}
	\frac{z}{1+z}\leqslant f(z)\leqslant \frac{z}{1-z}.
\end{equation}
We shall call them Seiffert functions. 

By $\Spint{S}$ we shall denote the set of these Seiffert functions for which both inequalities in \eqref{ineq:basic} are strict.

Note two important properties of Seiffert functions:
\begin{equation}
\lim_{z\to 0} f(z)=0, \text{ and } \lim_{z\to 0} \frac{f(z)}{z}=1.
\label{eq:propertiesofSeiffert}
\end{equation}

The set of all real functions $f:(0,1)\to\mathbb{R}$ satisfying  the condition $|f(z)|\leqslant 1$ for all $z$ will be denoted by  $\Sp{B}$ and its subset with strict inequality by $\Spint{B}$. The set $\Sp{B}$ is a complete metric space with metric $d_\Sp{B}(f,g)=\sup_{z}|f(z)-g(z)|$.

\section{Answers to the questions}

The next theorem gives complete (and rather surprising) answer to questions stated in the previous section. 
\begin{theorem}\label{thm:M_f_M} The mapping $f\to S_f$ is a ono-to-one correspondence between $\Sp{S}$ and $\Sp{M}$ that transforms $\Spint{S}$ onto $\Spint{M}$.
\end{theorem}
\begin{proof}
	It is obvious that $S_f$ is symmetric and positively homogeneous. $S_f$ is defined for all positive $x\neq y$, since $0<\frac{|x-y|}{x+y}<1$. Suppose that $x<y$ and let $z=\frac{y-x}{x+y}$. Then, the inequalities \eqref{ineq:basic} read
	$$\frac{{y-x}}{2y}\leqslant f\left(\frac{y-x}{x+y}\right)\leqslant \frac{{y-x}}{2x}.$$
	This is equivalent to $x\leqslant S_f(x,y)\leqslant y$ and shows that $S_f$ satisfies (C). Also note that for $f\in\Spint{S}$ the inequalities above are strict, which means that $S_f$ belongs to $\Spint{M}$.\\
	Conversely, for $M\in\Sp{M}$ we have
	\begin{align*}
	M(x,y)&=\frac{x+y}{2}M\left(\frac{x+y-(y-x)}{x+y},\frac{x+y+(y-x)}{x+y}\right)\\
	&=\frac{y-x}{2\dfrac{z}{M(1+z,1-z)}},
\end{align*}
and the function 
\begin{equation}
f_M(z)=\frac{z}{M(1+z,1-z)}
\label{eq:definition of f_M}
\end{equation} in the denominator of the right-hand side is a Seiffert function, because $1-z\leqslant M(1+z,1-z)\leqslant 1+z$.
Again, if $M$ is strict, then so is $f_M$.
\end{proof}
The next corollary follows immediately from the above proof. 
\begin{corollary}
	For arbitrary $M\in\Sp{M}$ and $f\in\Sp{S}$ the identities
	$$f_{S_f}(z)=z,\quad M(x,y)=S_{f_M}(x,y)$$
	hold.
\end{corollary}
Very important property of the mappings between the set of means and Seiffert functions is their antimonotonicity.
\begin{corollary}\label{cor:antimonotonicity}
The following conditions are equivalent
\begin{enumerate}[i]
	\item For all $x,y>0,\ M(x,y)\leqslant N(x,y),$
	\item For each $0<z<1, f_M(z)\geqslant f_N(z).$
\end{enumerate}
Also the following conditions are equivalent

\begin{enumerate}[i]\setcounter{enumi}{2}
	\item For all $x,y>0,\ x\neq y,\ M(x,y)< N(x,y),$
	\item For each $0<z<1, f_M(z)> f_N(z).$
\end{enumerate}
\end{corollary}
In many cases comparing the Seiffert funtions is much easier than comparing the means, so the antimonotonicity provides a new tool for proving inequalities between means.\\

The following corollary is a trivial consequence of Theorem \ref{thm:M_f_M}.

\begin{corollary} For the means $\min$ and $\max$ hold
	\begin{equation}
	f_{\min}(z)=\frac{z}{1-z},\qquad f_{\max}(z)=\frac{z}{1+z}.
	\label{eq:SF_minmax}
	\end{equation}
\end{corollary}

The pair of mappings $f\to S_f$ and $M\to f_M$  establishes a one-to-one and antimonotone with respect to corresponding partial orders, correspondence between the family of homogeneous symmetric means and the family $\Sp{S}$.\\
Let us transform \eqref{ineq:basic} as follows:
\begin{align}
	\frac{1-z}{z}&\leqslant \frac{1}{f(z)}\leqslant \frac{1+z}{z}\notag\\
	-1&\leqslant \frac{1}{f(z)} - \frac{1}{z} \leqslant 1. \label{ineq:basic1}
\end{align}
This enables us to define a metric on the set of Seiffert functions by
 \begin{equation}
d_S(f,g)=\sup_{0<z<1} \left|\frac{1}{f(z)}-\frac{1}{g(z)}\right|.
 \label{def:d_S}
 \end{equation} 
\begin{lemma}The space $(\Sp{S},d_s)$ has the following properties
	\begin{enumerate}
		\item \label{prop:p1}$(\Sp{S},d_s)$ is a complete metric space and $\diam\Sp{S}=2$,
		\item \label{prop:p2}$(\Sp{S},d_s)$ is a unit ball centered at the identity function $id(z)=z$.
	\end{enumerate}
\end{lemma}
\begin{proof}
	Completeness follows, since the convergence in $d_\Sp{S}$ implies the pointwise convergence, hence the limit function satisfies \eqref{ineq:basic1}. The same inequality implies (2).
\end{proof}

Clearly, the metric on $\Sp{S}$ induces the metric on $\Sp{M}$  as 
$$d_\Sp{M}(M,N)=d_\Sp{S}(f_M,f_N).$$
Thus our space of means is a unit ball centered at the arithmetic mean.\\
To obtain the explicit formula for $d_\Sp{M}(M,N)$, write
\begin{align*}\label{eq:}
	d_\Sp{M}(M,N)&=d_\Sp{S}(f_M,f_N)=\sup_{0<z<1} \left|\frac{1}{f_M(z)}-\frac{1}{f_N(z)}\right|	\\
	&=\sup_{x\neq y}\frac{2}{|x-y|}\left|\frac{|x-y|}{2f_M(z)}-\frac{|x-y|}{2f_N(z)}\right|=2\sup_{x\neq y}\left|\frac{M(x,y)-N(x,y)}{x-y}\right|.
\end{align*}
For more properties of this metric, see \cite{Far}.

As an application, let us prove the following result
\begin{theorem}\label{th:invmean}
	For  $M,N\in\Sp{M}$ satisfying $d_\Sp{M}(M,N)<2$, there exists a unique mean $K\in\Sp{M}$ such that for all $x,y$
	$$K(x,y)=K(M(x,y),N(x,y)).$$
\end{theorem}
Note that this result is known in case $M,N$ are strict means. There exist means that satisfy $d_\Sp{M}(M,N)<2$ and are not strict, and there are strict means with $d_\Sp{M}(M,N)=2$.
\begin{proof}
	Define the mapping $\Phi:\Sp{S}\to \Sp{S}$ by $\Phi(X)(x,y)=X(M(x,y),N(x,y)).$
	\begin{align*}
		\frac{1}{f_{\Phi(X)}(z)}&=\frac{X(M(1+z,1-z),N(1+z,1-z))}{z}=X\left(\frac{1}{f_M(z)},\frac{1}{f_N(z)}\right)
		\\		&=\frac{\frac{1}{f_M(z)}-\frac{1}{f_N(z)}}{2f_X\left(\frac{{f_M^{-1}(z)}-{f_N^{-1}(z)}}{{f_M^{-1}(z)}+{f_N^{-1}(z)}}\right)}=\frac{1}{2}\left(\frac{1}{f_M(z)}-\frac{1}{f_N(z)}\right)\frac{1}{f_X(u)},
	\end{align*}
	where $u=\frac{{f_N(z)}-{f_M(z)}}{{f_M(z)}+{f_N(z)}}.$ Thus
	$$\frac{1}{f_{\Phi(X)}(z)}-\frac{1}{f_{\Phi(Y)}(z)}=\frac{1}{2}\left(\frac{1}{f_M(z)}-\frac{1}{f_N(z)}\right)\left(\frac{1}{f_X(u)}-\frac{1}{f_Y(u)}\right),$$
	which yields
	$$d_\Sp{M}(\Phi(X),\Phi(Y))\leqslant \frac{1}{2}d_\Sp{M}(M,N)d_\Sp{M}(X,Y).$$
	Applying the Banach fixed point theorem, we complete the proof.
\end{proof}

\section{Group structure}
In this section we define   a group structure in $\Spint{M}$. Obviously we express the group action in terms of Seiffert functions. It follows from formula \eqref{ineq:basic} that the operator 
$\Sp{A}:\Spint{S}\to\Spint{B}$ defined by $$\Sp{A}(f)(z)=\frac{1}{f(z)}-\frac{1}{z}$$
is an isometry. Let $\gamma:(-1,1)\to\mathbb{R}$ be an odd, continuous bijection such that $\gamma(0)=0$. We define the addition on $\Spint{S}$ by
$$(f\oplus g)(z)=\Sp{A}^{-1}(\gamma^{-1}(\gamma(\Sp{A}(f)(z))+\gamma(\Sp{A}(g)(z)))).$$
One can easily check that $(\Spint{S},\oplus)$ is an abelian group, and the identity function is the neutral element.\\
The mapping $M\to f_M$  allows us to transfer the group structure to the set of strict means by setting
$$M\oplus N=S_{f_M\oplus f_N}.$$
Clearly, the arithmetic mean becomes the neutral element, and for every $M$ we have $M\oplus(2A-M)=A$.\\
 It is clear that if $g_n$ converges pointwise to $g$, then $f\oplus g_n$ converges pointwise to $f\oplus g$. Unfortunately, the addition is not continuous with respect to the metric $d_\Sp{S}$. To show this, consider the following example: for natural $n$ let $z_n=\gamma^{-1}(n)$ and let $g_n,g,f\in\Spint{B}$ be defined by
$$g_n(z)=\begin{cases}
	z_{n+1}&\text{ if } z=z_n,	\\
	z&\text{ otherwise,}	
\end{cases}\quad g(z)=z,\quad f(z)=-z.$$
We have $d_\Sp{B}(g_n,g)=|z_{n+1}-z_n|\to 0$, $\gamma^{-1}(\gamma(f(z))+\gamma(g(z)))=0$ and $\gamma^{-1}(\gamma(f(z_n))+\gamma(g(z_n)))=\gamma^{-1}(1).$ Consequently, $d_\Sp{M}(f\oplus g_n,f\oplus g)\not\to 0$. 

\section{New Seiffert-like means}\label{sec:new means}
Seiffert introduced two means corresponding to $\arcsin$ and $\arctan$. Two other means mentioned in the introduction come from their hyperbolic companions. Now, ewe shall show, that also $\sin,\;\tan,\;\sinh$ and $\tanh$ are Seiffert functions. 
To this end, we use two lemmae.
\begin{lemma}\label{new seiffert means > A}
	The inequalities
	\begin{equation}\label{main_ineq}
		t>\arsinh t> \arctan t>\tanh t >\frac{t}{1+t}
	\end{equation}
	hold for all $t>0$. Moreover, 
	\begin{equation}\label{ash>s}
		\arsinh t>\sin t
	\end{equation}
	 holds for $0<t<\pi/2$ and 
	 \begin{equation}\label{s>at}
		 \sin t>\arctan t
	 \end{equation}
	is valid 	for $0<t<1$.
\end{lemma}
\begin{proof}
	Since $\cosh^2t>\left(1+\frac{t^2}{2}\right)^2>1+t^2$,  integrating form $0$ to $t$ the inequalities
	$$1>\frac{1}{\sqrt{1+t^2}}>\frac{1}{1+t^2}>\frac{1}{\cosh^2 t},$$ we obtain first three inequalities in \eqref{main_ineq}. To prove the last one, observe that the graph of the convex function $\cosh t$ and the straight line $1+t$ intersect at two points: $t=0$ and $t=t_0>0$. Thus $\frac{1}{\cosh^2 t}-\frac{1}{(1+t)^2}$ is positive for $0<t<t_0$ and negative for $t>t_0$. Therefore, the function
	$$h(t)=\tanh t - \frac{t}{1+t}=\int_0^t \frac{1}{\cosh^2 t}-\frac{1}{(1+t)^2}dt$$ increases from $h(0)=0$ to $h(t_0)$ and then decreases to $h(\infty)=0$, hence is nonnegative. This completes the proof of the rightmost inequality in \eqref{main_ineq}.\\
To prove \eqref{ash>s}, note that for 
$0<t<\pi/2$ the inequalities 
\begin{align*}
(1+t^2)\cos^2 t&=(1+t^2)\frac{\cos 2t +1}{2}<(1+t^2)\frac{1-2t^2+2t^4/3 + 1}{2}\\
&=1-\tfrac{1}{3}t^4(2-t^2)<1
\end{align*}	
hold. Thus $\frac{1}{\sqrt{1+t^2}}>\cos t$ and we obtain \eqref{ash>s} by integration.\\
 To prove \eqref{s>at}, observe that for $0<t<1$
\begin{align*}
	(1+t^2)\cos t>(1+t^2)(1-t^2/2)=1+\frac{t^2(1-t^2)}{2}>1,
\end{align*}
and apply the same argument as above.
\end{proof}
Lemma \ref{new seiffert means > A} shows that the following inequalities hold if $x\neq y$
\begin{align}
\frac{x+y}{2}\leqslant M(x,y) \leqslant \frac{x-y}{2\sin\frac{x-y}{x+y}} \leqslant T(x,y) \leqslant  \frac{x-y}{2\tanh\frac{x-y}{x+y}}< \max(x,y).
\end{align}
Another set of means follows from the next lemma.
\begin{lemma}\label{new seiffert means < A}
	The inequalities
	\begin{equation}\label{main_ineq_2}
		t<\sinh t< \tan t<\artanh t <\frac{t}{1-t}
	\end{equation}
	hold for all $0<t<1$. Moreover, 
	\begin{equation}\label{sh<as}
		\sinh t<\arcsin t < \artanh t
	\end{equation}
	 holds for $0<t<1$.
	The functions $\arcsin t$ and $\tan t$
	 are not comparable in  $0<t<1$.
\end{lemma}
\begin{proof} The inequalities \eqref{main_ineq_2} follow from \eqref{main_ineq} and the fact, that the graph of an inverse function is symmetric with respect to the main diagonal. The same argument applied to \eqref{ash>s} implies the first inequality in \eqref{sh<as}, while the second inequality can be obtained by integration of $\frac{1}{\sqrt{1-t^2}}< \frac{1}{1-t^2}$.\\
It follows from \eqref{s>at} and the remark about the graph of an inverse function that $\arcsin t<\tan t$ for $t<\sin 1$, while $\arcsin 1>\tan 1$.
 
\end{proof}
Lemma \ref{new seiffert means < A} implies the following chains of inequalities between means
\begin{align}
\frac{x+y}{2}>\frac{x-y}{2\sinh \frac{x-y}{x+y} }> P(x,y) >L(x,y)>\min(x,y)\\
\frac{x+y}{2}>\frac{x-y}{2\sinh \frac{x-y}{x+y} } >\frac{x-y}{2\tan \frac{x-y}{x+y}}>  L(x,y)>\min(x,y).
\end{align}
So all hyperbolic and inverse offspring of sine and tangent functions forms Seiffert-like means. We shall see in a while that both of them are much more fertile. 

\section{Integral transformation}

\begin{theorem}\label{operator I}
If $f\in\Sp{M}$ is concave and 
\begin{align}
I(f)(z)=\int_0^z \frac{f(t)}{t}dt,
\end{align}
then $I(f)$ is also concave and $f(z)\leqslant I(f)(z)\leqslant z$ for all $z\in(0,1)$.\\
Similarly, if $f$ is convex, then $I(f)$ is also convex and $f(z)\geqslant I(f)(z)\geqslant z$.

\end{theorem}
\begin{proof}
First, note that setting $f(0)=0$ we extend $f$ to a concave function on $[0,1)$. Inequalities \eqref{ineq:basic} imply $f'(0)=1$ so $f(z)\leqslant z$ by concavity. Moreover, the divided difference $f(t)/t$ decreases which yields concavity of $I(f)$, and implies
$$f(z)=\int_0^z \frac{f(z)}{z}dt\leqslant \int_0^z \frac{f(t)}{t}dt\leqslant \int_0^z dt=z.$$
The proof in case of convex function is similar.
\end{proof}

The operator $I$ is monotone on the set of functions where it exists, and because $f_{\max}$ is concave and $f_{\min}$ is convex, we obtain the following corollary.
\begin{corollary}
If $f$ is a Seiffert function such that $I(f)$ exists, then $I(f)$ is also a Seiffert function.
\end{corollary}

In particular we conclude that the following functions are Seiffert functions
\begin{align*}
\Si_n(z)&=\begin{cases}\sin z & n=0\\\displaystyle\int_0^z \frac{\Si_{n-1}(t)}{t}dt & n>0\end{cases},\\
\ASi_n(z)&=\begin{cases}\arcsin z & n=0\\\displaystyle\int_0^z \frac{\ASi_{n-1}(t)}{t}dt & n>0\end{cases},\\
\Ti_n(z)&=\begin{cases}\tan z & n=0\\\displaystyle\int_0^z \frac{\Ti_{n-1}(t)}{t}dt & n>0\end{cases},\\
\ATi_n(z)&=\begin{cases}\arctan z & n=0\\\displaystyle\int_0^z \frac{\ATi_{n-1}(t)}{t}dt & n>0\end{cases},\\
\SHi_n(z)&=\begin{cases}\sinh z & n=0\\\displaystyle\int_0^z \frac{\SHi_{n-1}(t)}{t}dt & n>0\end{cases},\\
\ASHi_n(z)&=\begin{cases}\arsinh z & n=0\\\displaystyle\int_0^z \frac{\ASHi_{n-1}(t)}{t}dt & n>0\end{cases},\\
\THi_n(z)&=\begin{cases}\tanh z & n=0\\\displaystyle\int_0^z \frac{\THi_{n-1}(t)}{t}dt & n>0\end{cases},\\
\ATHi_n(z)&=\begin{cases}\artanh z & n=0\\\displaystyle\int_0^z \frac{\ATHi_{n-1}(t)}{t}dt & n>0\end{cases}.
\end{align*}
Since all functions mentioned in Lemma \ref{new seiffert means > A} are concave, we have the following grid of inequalities between means:
$$\begin{matrix}
	&\dots\ <&\frac{x-y}{2\ASHi_{2} \frac{x-y}{x+y}}&<&\frac{x-y}{2\ASHi_{1} \frac{x-y}{x+y}}&<&\frac{x-y}{2\arsinh \frac{x-y}{x+y}}&=M(x,y)\\
	&&\wedge&&\wedge&&\wedge&\\
	&\dots\ <&\frac{x-y}{2\Si_{2} \frac{x-y}{x+y}}&<&\frac{x-y}{2\Si_{1} \frac{x-y}{x+y}}&<&\frac{x-y}{2\sin \frac{x-y}{x+y}}&\\
	\frac{x+y}{2}<&&\wedge&&\wedge&&\wedge&\\
	&\dots\ <&\frac{x-y}{2\ATi_{2} \frac{x-y}{x+y}}&<&\frac{x-y}{2\ATi_{1} \frac{x-y}{x+y}}&<&\frac{x-y}{2\arctan \frac{x-y}{x+y}}&=T (x,y)\\
	&&\wedge&&\wedge&&\wedge&\\
	&\dots\ <&\frac{x-y}{2\THi_{2} \frac{x-y}{x+y}}&<&\frac{x-y}{2\THi_{1} \frac{x-y}{x+y}}&<&\frac{x-y}{2\tanh \frac{x-y}{x+y}}&\\
\end{matrix}.
$$
The  horizontal lines are granted by Theorem \ref{operator I} and vertical ones follow from Lemma \ref{new seiffert means > A}.

Before providing a similar picture for the other four functions, recall that there is no comparison between arcsine and tangent functions. Nevertheless, the operator $I$ quickly rectifies this irregularity.
\begin{lemma}\label{ASi<Ti}
For $0<z<1$ we have
$$\ASi_1(z)=\int_0^z \frac{\arcsin t}{t}dt < \int_0^z \frac{\tan t}{t}dt=\Ti_1(z).$$
\end{lemma}
\begin{proof}
Let $q(t)=\arcsin t - \tan t$. As shown in the proof of Lemma \ref{new seiffert means < A}, for $t<\sin 1\approx 0.841$ the inequality $q(t)<0$ holds. For $t>\pi/4\approx 0.785$  we have $q'(t)=\frac{\cos^2 t-\sqrt{1-t^2}}{\sqrt{1-t^2}\cos^2 t}$, $q'(\pi/4)<0$ and $q'(1)>0$. Since $\cos^2 t$ is convex and $\sqrt{1-t^2}$ concave, their graphs intersect exactly in one point. Thus $q(t)$ changes sign exactly once in the interval $(0,1)$. This implies that the function $u(z)=\int_0^z q(t)/t dt$ has exactly one local minimum, and since $u(z)=0$ and $u(1)\approx -0.016$, it is negative, which completes the proof.
\end{proof}
Now Theorem \ref{operator I} together with Lemmae \ref{new seiffert means < A} and \ref{ASi<Ti} yield
$$\begin{matrix}
	\frac{x+y}{2}> &\frac{x-y}{2\sinh \frac{x-y}{x+y}}&>&\frac{x-y}{2\SHi_{1} \frac{x-y}{x+y}}&>&\frac{x-y}{2\SHi_{2} \frac{x-y}{x+y}}&>\dots\\
	&\vee&&\vee&&\vee&\\
	P(x,y)= &\frac{x-y}{2\arcsin \frac{x-y}{x+y}}&>&\frac{x-y}{2\ASi_{1} \frac{x-y}{x+y}}&>&\frac{x-y}{2\ASi_{2} \frac{x-y}{x+y}}&>\dots\\
	&&&\vee&&\vee&\\
	\vee &\frac{x-y}{2\tan \frac{x-y}{x+y}}&>&\frac{x-y}{2\Ti_{1} \frac{x-y}{x+y}}&>&\frac{x-y}{2\Ti_{2} \frac{x-y}{x+y}}&>\dots\\
	&\vee&&\vee&&\vee&\\
	L(x,y)= &\frac{x-y}{2\artanh \frac{x-y}{x+y}}&>&\frac{x-y}{2\ATHi_{1} \frac{x-y}{x+y}}&>&\frac{x-y}{2\ATHi_{2} \frac{x-y}{x+y}}&>\dots\\
	&\vee&&\vee&&\vee&\\
\end{matrix}.
$$
\section{Schur convexity}
Given a symmetric, convex set $D\subset \mathbb{R}^2$, a partial order in $D$ is defined by
\begin{equation}\label{def:majoryzacja}
(x_1,y_1)\prec(x_2,y_2) \Leftrightarrow x_1+y_1=x_2+y_2 \text{ and } \max(x_1,y_1)\leqslant \max(x_2,y_2).
\end{equation}
A symmetric function $h:D\to\mathbb{R}$ is called Schur-convex if it preserves this partial order, i.e. if $(x_1,y_1)\prec(x_2,y_2)$ yields $h(x_1,y_1)\leqslant h(x_2,y_2)$, and Schur-concave if the partial order gets reversed.\\
Setting $c=(x_1+y_1)/2,\, t_1=|x_1-y_1|/2,\, t_2=|x_2-y_2|/2$, we see that $(x_1,y_1)\prec(x_2,y_2)$ is equivalent to $t_1\leqslant t_2$, and thus we can say that $h$ is Schur-convex (resp. concave) if and only if for all $c$ the function $h(c+t,c-t)$ increases (resp. decreases) for $t>0$, cf. \cite[I.3.A.2.b]{Mar}.
It will be useful to introduce the strict Schur-convexity: it is when the inequality in \eqref{def:majoryzacja} is strict whenever $(x_1,y_1)$ is not a permutation of $(x_2,y_2)$. In this case all reasoning in this section remains valid with the adverb 'strictly' added to all mentioned properties.

Schur convexity of means is an interesting subject investigated by many mathematicians (see e.g. \cite{ChuXia,Shi1,Cul} and the references therein).

 In case of a homogeneous symmetric mean, the Schur-convexity condition may be written in a very simple form: $M(x,y)$ is Schur-convex (resp. concave) if and only if the function $s(t)=M(1+t,1-t)$ increases in the unit interval. Let us see how this condition translates into the language of Seiffert functions. We have
$	s(t)=M(1+t,1-t)=\frac{t}{f_M(t)}	$, so we have the following.
\begin{theorem}
	A mean $M$ is Schur-convex (resp. concave) if and only if the function $f_M(z)/z$ decreases (res. increases).
\end{theorem}
Note that if a Seiffert function $f$ is concave (resp. convex), then its divided difference $f(z)/z$ decreases (resp. increases), so we have
\begin{corollary}
	If $f\in\Sp{S}$ is concave (resp. convex), then the mean $S_f(x,y)$ is Schur-convex (resp. Schur-concave).
\end{corollary}
\begin{corollary}
	The means
	$$\frac{x-y}{2\Si_n\frac{x-y}{x+y}},\quad \frac{x-y}{2\ASHi_n\frac{x-y}{x+y}},\quad \frac{x-y}{2\ATi_n\frac{x-y}{x+y}},\quad \frac{x-y}{2\THi_n\frac{x-y}{x+y}}$$
	are strictly Schur-convex, while means
	$$\frac{x-y}{2\Ti_n\frac{x-y}{x+y}},\quad \frac{x-y}{2\ATHi_n\frac{x-y}{x+y}},\quad \frac{x-y}{2\ASi_n\frac{x-y}{x+y}},\quad \frac{x-y}{2\SHi_n\frac{x-y}{x+y}}$$
	are strictly Schur-concave.	
	
\end{corollary}

\section{Means with varying arguments}\label{sec:Means with varying arguments}
Denote by $\widehat{f}(z)$ the function $f(z)/z$.
For $0<t<1, 0<z<1$ and $f\in \Sp{S}$ we have the inequalities
$$\frac{z}{1+z}<\frac{1}{t}\frac{tz}{1+tz}\leqslant \frac{f(tz)}{t}\leqslant \frac{1}{t}\frac{tz}{1-tz}<\frac{z}{1-z},$$
which shows that $f_t(z)=f(tz)/t$ are also Seiffert functions. Note that $\lim_{t\to 0} f_t(z)=z$, thus this process defines a homotopy between $f$ and $\id$. It is a matter of simple transformation to verify, that if $f$ corresponds to a mean $M$, then $f_t$ maps to the mean 
$$M_t(x,y)=M\left(\frac{x+y}{2}-t\frac{x-y}{2},\frac{x+y}{2}+t\frac{x-y}{2}\right).$$
There are numerous papers on comparison between means of the form $M_t$ and other classical means (see e.g. \cite{Chu0,Chu2,Chu4,Chu5,Gao,Neu,Wang1}). \\
The most popular problem is  formulated as follows: given two means $M,N$ satisfying $A<M<N$, find optimal $p,q$ such that $N_p\leqslant M\leqslant N_q$. To avoid double subscripts, let $m$ and $n$ be the Seiffert functions of $M$ and $N$, respectively. Then, by Corollary \ref{cor:antimonotonicity} the inequalities $N_p<M<N_q$ are equivalent to $n_p>m>n_q$. The inequality $A<N$ implies $n(z)<z$.   Assume additionally that the function $\widehat{n}(z)$ is strictly decreasing, (in case of classical means,  their Seiffert functions are usually concave, so this condition is satisfied). Then the following inequalities are equivalent:
$$n_p(z)>m(z)>n_q(z)\equiv\frac{n(pz)}{pz}>\frac{m(z)}{z}>\frac{n(qz)}{qz}\equiv p<\frac{\widehat{n}^{-1}(\widehat{m}(z))}{z}<q.$$
Thus we have proven the following theorem.
\begin{theorem}
	Let $M$ and $N$ be two means with Seiffert functions $m$ and $n$, respectively. Suppose that $\widehat{n}(z)$ is strictly monotone and let $p_0=\inf\limits_z \frac{\widehat{n}^{-1}(\widehat{m}(z))}{z}$ and $q_0=\sup\limits_z\frac{\widehat{n}^{-1}(\widehat{m}(z))}{z}$. \\
	If $A(x,y)<M(x,y)<N(x,y)$ for all $x\neq y$ then the inequalities
		$$N_p(x,y)\leqslant M(x,y)\leqslant N_q(x,y)$$
	hold if and only if $p\leqslant p_0$ and $q\geqslant q_0$.\\
	If $N(x,y)<M(x,y)<A(x,y)$ for all $x\neq y$ then the inequalities
		$$N_q(x,y)\leqslant M(x,y)\leqslant N_p(x,y)$$
	hold if and only if $p\leqslant p_0$ and $q\geqslant q_0$.
\end{theorem}

To illustrate this theorem, let us consider two examples featuring the contraharmonic mean $C(x,y)=\frac{x^2+y^2}{x+y}$.
\begin{example}
	Let $M(x,y)=\frac{x^2+xy+y^2}{x+\sqrt{xy}+y}$. It is known (cf. \cite{AW1}) that for $x\neq y$ 
	$$A(x,y)<M(x,y)<C(x,y)$$
	holds. We have
	$$\widehat{m}(z)=\frac{2+\sqrt{1-z^2}}{3+z^2},\ \widehat{c}(z)=\frac{1}{1+z^2},\ \frac{\widehat{c}^{-1}(\widehat{m}(z))}{z}=\sqrt{\frac{1-\sqrt{1-z^2}}{z^2}}.$$
	The function $1-\sqrt{1-u}$ is convex, so its divided difference increases from  $1/2$ to $1$, therefore we obtain
	$$C_{\sqrt{2}/2} (x,y)\leqslant M(x,y)\leqslant C(x,y).$$
\end{example}
\begin{example}\label{ex:AGC}
	The contraharmonic mean belongs to the family of Gini means defined in general case by $G(r,s;x,y)=\left(\frac{x^r+y^r}{x^s+y^s}\right)^{1/(r-s)}$. They are increasing with respect to parameters $r,s$, thus for $0<\alpha<2$ we have
	$$A(x,y)=G(1,0;x,y)<G(1,\alpha;x,y)<G(1,2;x,y)=C(x,y).$$
	Fix $\alpha$ and let $M=G(1,\alpha)$. An easy calculation shows that
	$$\frac{\widehat{c}^{-1}(\widehat{m}(z))}{z}=2^{-\frac{1}{2(\alpha-1)}}\sqrt{\frac{\left[{(1+z)^\alpha+(1-z)^\alpha}\right]^{1/(\alpha-1)}-2^{1/(\alpha-1)}}{z^2}}.$$
	We shall show that the function under the square root increases. To this end, we shall use the following version of de l'Hospital's rule (\cite{Pin}): if $f$ and $g$ are differentiable functions  $g'(x)\neq 0$ and such that $(f'/g')(x)$ increases (decreases), then so does the divided difference $\frac{f(x)-f(a)}{g(x)-g(a)}$.\\
	Let $f(x)=\left[{(1+z)^\alpha+(1-z)^\alpha}\right]^{1/(\alpha-1)}$ and $g(x)=z^2$. We intend to show that $(f'/g')(z)$ increases in $(0,1)$. We have
	\begin{align*}\label{eq:}
		\frac{f'}{g'}(z)&	=\left[{(1+z)^\alpha+(1-z)^\alpha}\right]^{(2-\alpha)/(\alpha-1)}\times \frac{\alpha}{2(\alpha-1)}\frac{{(1+z)^{\alpha-1}-(1-z)^{\alpha-1}}}{z}\\
		&	=: h_1(z)\times \frac{\alpha}{2(\alpha-1)}\frac{h_2(z)}{z}.
	\end{align*}
	Case 1: $1<\alpha<2$. The power function $z^\alpha$ is convex, so $(1+z)^\alpha+(1-z)^\alpha$ increases and so does $h_1(z)$. The function $h_2(z)$ is positive and convex, so its divided difference increases, thus $f'/g'$ increases being a product of positive increasing functions.
	
	\noindent Case 2: $0<\alpha<1$. The power function $z^\alpha$ is concave, so $(1+z)^\alpha+(1-z)^\alpha$ decreses, thus $h_1(z)$ increases. The function $h_2(z)$ is negative and concave, so its divided difference decreases, and the negative factor $\frac{\alpha}{2(\alpha-1)}$ turns it into positive increasing. Thus again $f'/g'$ increases.
	
Therefore 
	$$\sqrt{\frac{\alpha}{2}}= \lim_{z\to 0}\frac{\widehat{c}^{-1}(\widehat{m}(z))}{z}\leqslant \frac{\widehat{c}^{-1}(\widehat{m}(z))}{z}\leqslant \frac{\widehat{c}^{-1}(\widehat{m}(1))}{1}=1$$
and we obtain the optimal inequalities
$$C\left(\frac{x+y}{2}+\sqrt{\frac{\alpha}{2}}\frac{x-y}{2},\frac{x+y}{2}-\sqrt{\frac{\alpha}{2}}\frac{x-y}{2}\right)\leqslant G(1,\alpha;x,y)\leqslant C(x,y).$$
	
\end{example}
Using the same technique as in Example \ref{ex:AGC} we obtain the following results for power means $G(0,\alpha;x,y)$:
\begin{example}
Let $N(x,y)=G(0,2;x,y)$ denote the root-mean square. For $1<\alpha<2$ the following  inequalities 
$$N\left(\tfrac{x+y}{2}+p\tfrac{x-y}{2},\tfrac{x+y}{2}-p\tfrac{x-y}{2}\right)\leqslant G(0,\alpha;x,y)\leqslant N\left(\tfrac{x+y}{2}+q\tfrac{x-y}{2},\tfrac{x+y}{2}-q\tfrac{x-y}{2}\right)$$
hold if and only if $p\leqslant\sqrt{\alpha-1}$ and $q\geqslant \sqrt{4^{1-1/\alpha}-1}$.
\end{example}

\begin{example}
For $-1<\alpha<1$ the following  inequalities 
$$H(x,y)\leqslant G(0,\alpha;x,y)\leqslant H\left(\frac{x+y}{2}+q\frac{x-y}{2},\frac{x+y}{2}-q\frac{x-y}{2}\right)$$
hold if and only if   $q\leqslant \sqrt{\frac{1-\alpha}{2}}$.
\end{example}

\section{Aproximation by convex combination of means}
Suppose all positive distinct $x,y$three means $K,M,N$ satisfy for  inequalities $K(x,y)< M(x,y)< N(x,y)$. Our goal is to determine the best possible constants $\mu$ and $\nu$ such that the inequalities 
 \begin{equation}
(1-\mu)K(x,y)+\mu N(x,y)\leqslant M(x,y)\leqslant (1-\nu)K(x,y)+\nu N(x,y)
 \label{ineq:convex combination}
 \end{equation} are valid for all $x,y$.
In terms of  Seiffert functions, the inequalities \eqref{ineq:convex combination} have the form
$$\frac{1-\mu}{f_K(z)}+\frac{\mu}{f_N(z)}\leqslant \frac{1}{f_M(z)}\leqslant \frac{1-\nu}{f_K(z)}+\frac{\nu}{f_N(z)},$$
which is equivalent to
\begin{equation}
\mu\leqslant \frac{\frac{1}{f_M(z)}-\frac{1}{f_K(z)}}{\frac{1}{f_N(z)}-\frac{1}{f_K(z)}}\leqslant \nu.
\label{ineq:bounds for convex combination}
\end{equation}
Thus we have the following result.
\begin{theorem}
For three means satisfying $K<M<N$ the inequalities \eqref{ineq:convex combination} hold if and only if 
$$\mu\leqslant \inf_{0<z<1}\frac{\frac{1}{f_M(z)}-\frac{1}{f_K(z)}}{\frac{1}{f_N(z)}-\frac{1}{f_K(z)}} \text{ and } 
\sup_{0<z<1}\frac{\frac{1}{f_M(z)}-\frac{1}{f_K(z)}}{\frac{1}{f_N(z)}-\frac{1}{f_K(z)}}\leqslant \nu.$$
\end{theorem}
Let us illustrate this result by finding the optimal bounds of eight Seiffert-like means discussed in Section \ref{sec:new means} by convex combination of $\min$ and $\max$. In this case the inequalities \eqref{ineq:bounds for convex combination} read 
$$\mu\leqslant \frac{1}{2}\left(\frac{1}{f(z)}-\frac{1}{z}+1\right)\leqslant \nu,$$ 
thus we have to find the upper and lower bound of the function $\frac{1}{f(z)}-\frac{1}{z}$. In all eight cases we can write $f(z)=z+cz^3+\dots$, so $\lim_{z\to 0}\frac{1}{f(z)}-\frac{1}{z}=0$.
\begin{example}\label{ex:s}
	Let $f(z)=\sin z$. The function $\cos z$ is concave in $(0,\pi/2)$ thus so are $\cos(tz)$ and $\int_0^1\cos(tz)dt=\frac{\sin z}{z}$. Since the reciprocal of a positive concave function is convex,  $\frac{z}{\sin z}$ is convex and consequetly, its divided difference 
	$$\frac{1}{z}\left(\frac{z}{\sin z}-1\right)=\frac{1}{\sin z}-\frac{1}{z}$$
	increases from $0$ to $1/\sin(1) -1$. Thus
	$$\frac{x+y}{2}\leqslant \frac{x-y}{2\sin\frac{x-y}{x+y}}\leqslant \left(1-\frac{1}{2\sin 1}\right)\min(x,y)+\frac{1}{2\sin 1}\max(x,y).$$
\end{example}
\begin{example}
Let $f(z)=\arcsin z$. To investigate monotonicity of $\frac{1}{\arcsin z}-\frac{1}{z}$, set $z=\sin t$ and use the result of the previous example  to see, that the range of this function is $(2/\pi -1, 0)$. Thus,
$$\frac{1}{\pi}\min(x,y)+\left(1-\frac{1}{\pi}\right)\max(x,y)\leqslant P(x,y)\leqslant \frac{x+y}{2}.$$
\end{example}
\begin{example}
	Let $f(z)=\tan z$. Then,
	\begin{align*}
	\left(\frac{1}{\tan z}-\frac{1}{z}\right)'	&=\frac{1}{z^2}-\frac{\tan'(z)}{\tan^2(z)}=	\frac{\sin^2 z - z^2}{z^2\sin^2 z}<0,
	\end{align*} so $\frac{1}{\tan z}-\frac{1}{z}$ decreases from $0$ to $1/\tan(1)-1$  and consequently
	
	$$\left(1-\frac{1}{2\tan 1}\right)\min(x,y)+\frac{1}{2\tan 1}\max(x,y)\leqslant \frac{x-y}{2\tan\frac{x-y}{x+y}}\leqslant \frac{x+y}{2} .$$
\end{example}
\begin{example}
	In case $f(z)=\arctan z$, we set $z=\tan t$ and use the previous example to get
	$$\frac{x+y}{2}\leqslant T(x,y)\leqslant \left(1-\frac{2}{\pi}\right)\min(x,y)+\frac{2}{\pi}\max(x,y).$$
\end{example}
\begin{example}\label{ex:ht}
	The story of hyperbolic tangent is quite similar to that of tangent, the only difference is that $\sinh z\geqslant z$. 
	$$\frac{x+y}{2}\leqslant \frac{x-y}{2\tanh\frac{x-y}{x+y}}\leqslant  \left(1-\frac{1}{2\tanh 1}\right)\min(x,y)+\frac{1}{2\tanh 1}\max(x,y).$$
\end{example}
\begin{example}
As shall be expected, for the inverse hyperbolic tangent we shall use Example \ref{ex:ht} and the substitution  $t=\tanh z$. Nevertheless, it is wise to note that $\lim_{z\to 1} \artanh z=\infty$, so in this case there is only trivial lower bound
	$$\min(x,y)\leqslant L(x,y)\leqslant \frac{x+y}{2}.$$
\end{example}
\begin{example}
	An attempt to use the same approach as in Example \ref{ex:s} fails for $f(z)=\sinh z$, because $\sinh z/z$ is convex and the reciprocal of a convex function may not be convex.
\begin{align}\label{eq:sinh}
\left(\frac{1}{\sinh z}-\frac{1}{z}\right)'	&=\frac{\sinh^2 z - z^2\cosh z}{z^2\sinh^2 z}\\
	&	=\frac{\cosh z}{\sinh^2 z}\left(\frac{\sinh z \cosh^{-1/2}z}{z}+1\right)\left(\frac{\sinh z \cosh^{-1/2}z}{z}-1\right)\notag.
\end{align}	
Let $q(z)=\sinh z \cosh^{-1/2}z$. Then $q''(z)=\frac{1}{4}\cosh^{-3/5}z\sinh z(\cosh^2z -3)$, hence $q$ is concave in $(0,\arcosh \sqrt{3})$. Since $\arcosh \sqrt{3}\approx 1.14$, we conclude that $q(z)/z$ decreases in the interval $(0,1)$, and thus the expression \eqref{eq:sinh} is negative. Therefore we have 
$$\left(1-\frac{1}{2\sinh 1}\right)\min(x,y)+\frac{1}{2\sinh 1}\max(x,y)\leqslant \frac{x-y}{2\tan\frac{x-y}{x+y}}\leqslant \frac{x+y}{2} .$$
\end{example}
And finally
\begin{example}
	For $f(z)=\arsinh z$, the substitution $z=\sinh t$ converts $\frac{1}{\arsinh z}-\frac{1}{z}$ into $\frac{1}{t}-\frac{1}{\sinh t}$ and monotonicity follows from the previous example, since $1>z>t$. The optimal inequalities then are
	$$\frac{x+y}{2}\leqslant \frac{x-y}{2\arsinh\frac{x-y}{x+y}}\leqslant  \left(1-\frac{1}{2\arsinh 1}\right)\min(x,y)+\frac{1}{2\arsinh 1}\max(x,y).$$
\end{example}

\section{Miscellanea}
In this section, we collect some facts about means and Seiffert functions, that might be useful for future investigations.

This surprising result follows from the results of Section \ref{sec:Means with varying arguments}.
\begin{theorem}
	If $M$ is a mean satisfying $M(x,y)\leq A(x,y)$, and $M_{1/2}(x,y)=M(\frac{3x+y}{4},\frac{x+3y}{4})$, then
	$M_{1/2}^2/A$ is a mean.
\end{theorem}
Note that in general $M^2/A$ is not a mean (take the harmonic mean as an counterexample).
\begin{proof}
	Let $f$ be the Seiffert mean of $M$. Then the function $2f(z/2)$ is the Seiffert function of $M_{1/2}$. Consider the function $g(z)=4f^2(z/2)/z$. Since $z\leqslant f(z)$, we have $z\leqslant g(z)$ and 
	$$g(z)\leqslant \frac{4}{z}\frac{z^2/4}{(1-z/2)^2}<\frac{z}{1-z}.$$
Thus $g$ is  a Seiffert function and its corresponding mean is
$$\frac{|x-y|z}{8f^2(z/2)}=\left(\frac{|x-y|}{2\cdot 2f(z/2)}\right)^2\frac{2}{x+y}=\frac{M_{1/2}^2(x,y)}{A(x,y)}.$$
\end{proof}
This result may be generalized as follows
\begin{corollary}
	If $M$ is a mean satisfying $M(x,y)\leq A(x,y)$, and $M_{t}(x,y)=M(\frac{x+y}{2}+t\frac{x-y}{2},\frac{x+y}{2}-t\frac{x-y}{2})$, then
	$M_{t}^{1/t}A^{1-1/t}$ is also a mean.
\end{corollary}
To prove this, it is enough to show that $g(z)=z[g(tz)/(tz)]^{1/t}$ is a Seiffert function. We leave the details to the reader.
\begin{remark}\label{rem:z^3}
It is easy to see that $z+az^3$ is a Seiffert function if and only if $-1/2\leqslant a\leqslant 1/2$.
\end{remark}
\begin{theorem}
	The inequalities
	\begin{equation}
	A(x,y)\leq \frac{|x-y|}{2\sin \frac{x-y}{x+y}}\leq A(x,y)\frac{6A^2(x,y)}{5A^2(x,y)+G^2(x,y)}
	\label{ineq:sin}
	\end{equation}
	hold.
\end{theorem}
\begin{proof}
	The sine function satisfies $z>\sin z>z-z^3/6$. It follows from Remark \ref{rem:z^3} that the last function in this chain is a Seiffert function, and a simple calculation shows that its mean is the rightmost mean in \eqref{ineq:sin}.
\end{proof}

For two Seiffert functions $f_M$ and $f_N$ and a homogeneous but necessarily symmetric mean $K$ , the function $g(z)=K(f_M(z),f_N(z))$ is  a Seiffert function  corresponding to $MN/K(N,M)$. In particular, if $K$ is $M,N$ invariant (see Theorem \ref{th:invmean}), then $MN/K$ is also a mean. Thus the weighted arithmetic mean of two Seiffert functions corresponds to the weighted harmonic mean of means and vice versa.

Now we shall show facts about power series representation of Seiffert functions. The first one is trivial.
\begin{theorem}
	If $0\leqslant a_n\leqslant 1$ for $n>1$, then the function $f(z)=z+\sum_{n=2}^\infty a_nz^n$ is  a Seiffert function.
\end{theorem}
\begin{proof}
	Clearly, the series converges for $0<z<1$ and 
	$$z\leqslant f(z)\leqslant \sum_{n=1}^\infty z^n=\frac{1}{1-z}.$$
\end{proof}
The two following theorems concern alternating series.
\begin{theorem}
	Let $1=a_1\geqslant a_2\geqslant \dots\geqslant 0$ be a convex sequence (i.e. satisfying $2a_k\leqslant a_{k-1}+a_{k+1}$ for $k=2,3,\dots$). Then the function $f(z)=z+\sum_{n=2}^\infty (-1)^{n+1}a_nz^n$ is a Seiffert function.
\end{theorem}
\begin{proof}
	Note first, that
	$$f(z)=z-z^2(a_2-a_3z)-z^4(a_4-a_5z)-\dots\leqslant z.$$
	Let $b_n=a_n-a_{n+1}$. This sequence decreases monotonically to $0$ and we have
	\begin{align*}\label{eq:}
		(1+z)f(z)&=z+\sum_{n=2}^\infty (-1)^{n+1}a_nz^n+\sum_{n=1}^\infty (-1)^{n+1}a_nz^{n+1}	\\
		&	=z+z^2(b_1-b_2z)+z^4(b_3-b_4z)+\dots\geqslant z,
	\end{align*}
	thus $f(z)\geqslant \frac{z}{1+z}$.
\end{proof}

\begin{corollary}
	The following functions are Seiffert:
	$$\log(1+z),\quad 2\frac{z-\log(1+z)}{z},\quad 3\frac{\log(1+z)-z+z^2/2}{z^2},\ \dots.$$
\end{corollary}
\begin{theorem}
	Let $a_n,\ n\geqslant 1$ be nonnegative numbers satisfying $a_1\leqslant 1/2, 1\geqslant a_2\geqslant a_3\geqslant\dots$, then $f(z)=z-a_1z^3+\sum_{n=2}^\infty (-1)^n a_nz^{2n+1}$ is a Seiffert function.
\end{theorem}
\begin{proof}
	As we know from Remark \ref{rem:z^3} that $z-a_1z^3\geqslant \frac{z}{1+z}$, thus
	$$f(z)\geqslant \frac{z}{1+z}+z^5(a_2-a_3z^2)+z^9(a_4-a_5z^2)+\dots\geqslant \frac{z}{1+z}.$$
	On the other hand
	$$f(z)\leqslant \sum_{n=0}^\infty z^{2n+1}< \sum_{n=1}^\infty z^{n}=\frac{z}{1-z},$$
	which completes the proof.
\end{proof}
\begin{corollary}
	The following are Seiffert functions:
	$$\sin z,\quad 6\frac{z-\sin z}{z^2},\quad 120\frac{\sin z -z +z^3/6}{z^4},\ \dots,$$
	$$2\frac{1-\cos z}{z},\quad 24\frac{\cos z - 1 +z^2/z}{z^3},\ \dots .$$
\end{corollary}

\end{document}